\documentclass[12pt]{article}

\usepackage{authblk} 
\usepackage[english]{babel}

\usepackage[a4paper,margin=1in]{geometry}
\usepackage{amsmath,amssymb,amsthm,bm,mathabx}
\usepackage{mathtools}
\usepackage{hyperref, color}
\usepackage{comment}

\newcommand{\PP}{\mathbb{P}}
\newcommand{\RR}{\mathbb{R}}
\newcommand{\EE}{\mathbb{E}}

\newtheorem{theorem}{Theorem}[section]	
\newtheorem{lemma}{Lemma}[section]
\newtheorem{corollary}{Corollary}[section]

\newcommand\td{\overset{d}{\to}}

\begin{document}

\title{Joint Distributions of Minimum and Maximum Angles on High-Dimensional Spheres}
\author[]{Yongcheng Qi}
\author[]{Lijian Yang}
 
\affil{University of Minnesota Duluth and Tsinghua University}

\date{}

\maketitle

\begin{abstract} Consider $n$ independent random vectors sampled from uniform distribution on $(p-1)$-dimensional unit sphere.
This paper investigates the limiting joint distribution for the minimum and the maximum values of their pairwise angles.  It proves that  the minimum and the maximum angles are asymptotically independent when both $n$ and $p$ tend to infinity, which solves an open problem raised in Cai, Fan and Jiang (2013) [\emph{Journal of Machine Learning Research} 14, 1837-1864]. Cai, Fan and Jiang (2013) obtained the limiting marginal distributions for both the minimum and the maximum angles under assumption $\lim_{n\to\infty}{\ln n}/p=\beta$ according to whether $\beta=0$,  
$\beta\in (0,\infty)$, or $\beta=\infty$.  This paper presents  unified limits for both joint distributions and marginal distributions regardless of the relative divergence rate of $n$ and $p$. The paper also derives the limiting distributions for some statistics based on the minimum and the maximum angles,     
\end{abstract}


\thispagestyle{empty}

\section{Introduction}\label{intro}

Random points on high-dimensional spheres arise naturally in a wide range of
problems in probability, statistics, geometry, physics, and machine learning. 
An important statistical application is the problem of testing whether a
collection of directional observations is uniformly distributed on the unit sphere
\[
\mathbb S^{p-1}
=
\{x\in\mathbb R^p:\|x\|=1\},
\]
where $\|x\|$ denotes the Euclidean norm of vector $x$.
Given $n$ observations
$X_1,\ldots,X_n\in\mathbb S^{p-1}$ from a population $X$,
the null hypothesis is
\[
H_0: X \text{ is uniformly distributed over } \mathbb S^{p-1}
(\text{notation:}\text{Unif}(\mathbb S^{p-1})).
\]
This is a fundamental goodness-of-fit problem in directional statistics.
Applications arise whenever observations represent directions, orientations,
or normalized vectors. Examples include geological and biological directional
data, astronomy, computer vision, and high-dimensional statistical data.

Classical tests of spherical uniformity include the Rayleigh, Bingham, Gine,
and Ajne tests. These procedures are based on different aspects of the
observations. For example, the Rayleigh test is sensitive to alternatives
having a nonzero mean direction, whereas the Bingham test is designed to
detect departures from uniformity through second-order directional structure.
A substantial literature has developed around such tests; see, for example,
\cite{MardiaJupp2000,Jupp2001,Figueiredo2007}.

In this paper, we assume $X_1,\ldots, X_n$ is a random sample of size $n$ from  $\text{Unif}(\mathbb S^{p-1})$ distribution
and consider the pairwise angles
\[
\Theta_{ij}
=
\arccos(X_i^T  X_j),
\qquad 1\leq i<j\leq n.
\]
The collection $\{\Theta_{ij}:1\leq i<j\leq n\}$ 
describes the geometric configuration of the random points through their
mutual directions. Unlike coordinate-based characteristics, pairwise angles
are invariant under rotations of the entire point configuration and therefore
provide a natural tool for studying spherical data.

The distribution of the angle between two independent uniformly distributed
points on $\mathbb S^{p-1}$ is explicitly known. In particular, its probability density function is given by 
\[
f_p(\theta)
=
\frac{\Gamma(p/2)}
{\sqrt{\pi}\,\Gamma((p-1)/2)}
(\sin\theta)^{p-2},
\qquad 0\leq\theta\leq\pi.
\]
Consequently, when $p$ is fixed, the empirical distribution of the
$\binom{n}{2}$ pairwise angles
\[\mu_n=\frac{1}{\binom{n}{2}}\sum_{1\le i<j\le n}\delta_{\Theta_{ij}}\]
converges weakly to the measure $\nu$ with density $f_p(\theta)$. See, e.g. Cai et al.~\cite{CFJ13}. 
The situation becomes
especially interesting when the dimension $p$ also increases. In that regime,
the angles become concentrated around $\pi/2$, reflecting the familiar
phenomenon that two independent high-dimensional random vectors are nearly orthogonal. 
Cai et al.~\cite{CFJ13} considered the normalized empirical measure
\[\mu_{n,p}=\frac{1}{\binom{n}{2}}\sum_{1\le i<j\le n}\delta_{\sqrt{p-2}(\tfrac{\pi}{2}-\Theta_{ij})}\]
and showed that it converges weakly to the standard normal distribution.

The study of pairwise angles is also closely connected with random packing
problems on spheres. Given $n$ points on $\mathbb S^{p-1}$, the smallest
pairwise angle and the largest
pairwise angle, denoted as  
\begin{equation}\label{minmax}
    \Theta_{\min}
=
\min_{1\le i<j\le n}\Theta_{ij}, ~~~\Theta_{\max}
=
\max_{1\le i<j\le n}\Theta_{ij},
\end{equation}
correspond, respectively, to the most nearly coincident pair and the most nearly antipodal pair.
They quantify different aspects of the geometry of a random
point cloud. Such quantities are relevant not only to stochastic geometry
but also to coding theory, spherical designs, nearest-neighbor methods, and
high-dimensional data analysis.
An unusually small minimum angle
may indicate clustering of observations, whereas an unusually large maximum
angle may indicate an excess of nearly antipodal observations. Thus, tests
based on extreme pairwise angles can be viewed as tests for particular forms
of departure from spherical uniformity.

Cai et al.~\cite{CFJ13} studied the asymptotic behaviors of the extreme  pairwise angles when $X_1,\ldots, X_n$ are  independently sampled from
$\text{Unif}(\mathbb S^{p-1})$ distribution in two regimes: (i) fixed dimension $p$ as $n\to\infty$, and (ii) both $n,p\to\infty$. The authors derived the limiting distributions of $\Theta_{\min}$ and $\Theta_{\max}$. In particular, when both $p$ and $n$ tend to infinity,  Cai et al.~\cite{CFJ13}  obtained the limiting distributions for some  cosine function or sine function of the extreme angles under condition  $\lim_{n\to\infty}\ln n/p=\beta$ for three different settings:  $\beta=0$, $\beta\in (0,\infty)$, and $\beta=\infty$.
Cai et al.~\cite{CFJ13} also showed that  $\Theta_{\min}$ and $\Theta_{\max}$  are asymptotically independent when $p$ is fixed as $n$ tends to infinity.  Meanwhile, they posed an open problem on the asymptotic independence of  $\Theta_{\min}$ and $\Theta_{\max}$  when $p$ diverges as $n$ tends to infinity. The limiting joint distribution of $\Theta_{\min}$ and $\Theta_{\max}$  remains unknown in the literature.

Our objective in this paper is to investigate the limiting joint distribution of  
 $\Theta_{\min}$ and $\Theta_{\max}$ under the general setting that both of $p$ and $n$ tend to infinity  
 regardless of the convergence rate of $p$ relative to $n$.  We will show $\Theta_{\min}$ and $\Theta_{\max}$ are asymptotically independent, and the limiting marginal distributions are Gumbel or reversed Gumbel. We solve the open problem raised in Cai et al.~\cite{CFJ13}.   

The rest of the paper is organized as follows.  In section 2,  we introduce the main results of the paper, including unified limiting joint distributions for $\Theta_{\min}$ and $\Theta_{\max}$ and limiting distributions for a few statistics as functions of $\Theta_{\min}$ and $\Theta_{\max}$. In section~\ref{proofs}, we present some lemmas and proofs for the main results.

\section{Main Results}\label{main}

Throughout this section, we assume
$X_1,\ldots,X_n$ are i.i.d.  random vectors with $\text{Unif}(\mathbb S^{p-1})$ distribution, and $p=p_n\to\infty$ as $n\to\infty$.  The minimum and maximum angles, $\Theta_{\min}$  and $\Theta_{\max}$, are defined as in \eqref{minmax}.

We  define the following cumulative distribution function
\[ 
\Lambda(x)=\exp\left\{-\frac{1}{2\sqrt{2\pi}}e^{-x/2}\right\}, ~~x\in \RR,
\]
which is a Gumbel distribution.   Obviously,
\begin{equation}\label{D}
D(x)=1-\Lambda(-x)=1-\exp\left\{-\frac{1}{2\sqrt{2\pi}}e^{-x/2}\right\}, ~~x\in \RR
\end{equation}
is also a cumulative distribution function.  $D(x)$ is a reversed Gumbel distribution sine it is the cumulative distribution function of a random variable $Z$ when $-Z$ has the Gumbel distribution $\Lambda(x)$.  

For convenience, some parameters are denoted as follows
\begin{equation}\label{alpha-beta}
t_n=\exp\left\{-\frac{2\ln n}{p_n-1}\right\}, ~~
    \alpha_n=\frac{t_n}{2(p_n-1)\sqrt{1-t_n^2}}, ~~\beta_n=\alpha_n\ln (p_n-1)(1-t_n^2).
\end{equation}

Our main result on the limiting distribution of  $\Theta_{\min}$  and $\Theta_{\max}$ is the following theorem.  

\begin{theorem}\label{thm1} Assume $p=p_n\to\infty$ as $n\to\infty$. Set
\begin{equation}\label{thetamin}
\widebar{\Theta}_{\min}=\frac{1}{\alpha_n}\left\{\Theta_{\min}-\arcsin(t_n)\right\}-\ln (p_n-1)(1-t_n^2)
\end{equation}
and
\begin{equation}\label{thetamax}
\widebar{\Theta}_{\max}=\frac{1}{\alpha_n}\left\{\pi-\Theta_{\max}-\arcsin(t_n)\right\}-\ln (p_n-1)(1-t_n^2).
\end{equation}
Then $(\widebar{\Theta}_{\min}, \widebar{\Theta}_{\max})$ converges in distribution to $(Z_1,Z_2)$, where $Z_1$ and $Z_2$ are independent random variables, each with a reversed Gumbel distribution $G(x)$ as given in \eqref{D}.
\end{theorem}

The following corollary is a direct consequence of Theorem~\ref{thm1}.

\begin{corollary}\label{cor1}  Assume $p=p_n\to\infty$ as $n\to\infty$.

\noindent(a). As $n\to\infty$,
\[
\frac{1}{\alpha_n}\big(\Theta_{\max}+\Theta_{\min}-\pi\big)\overset{d}\to Z_2-Z_1,
\]
where $\td$ stands for convergence in distribution.

\noindent(b). As $n\to\infty$, both  $\Theta_{\min}-\arcsin(t_n)$ and $\pi-\Theta_{\max}-\arcsin(t_n)$ converge in probability to zero.  
\end{corollary}

As an application of Theorem~\ref{thm1}, we present the limiting distributions for $\cos(\Theta_{\min})$ and $\cos(\Theta_{\max})$.

\begin{theorem}\label{thm2}
Assume $p=p_n\to\infty$ as $n\to\infty$. Define
\[
\widebar{C}_{\min}=\frac{1}{\alpha_nt_n}\left\{\cos(\Theta_{\min})
-\sqrt{1-t_n^2}\right\}+\ln (p_n-1)(1-t_n^2)
\]
and
\[
\widebar{C}_{\max}=\frac{1}{\alpha_nt_n}\left\{\cos(\Theta_{\max})
+\sqrt{1-t_n^2}\right\}-\ln (p_n-1)(1-t_n^2).
\]
Then $\widebar{C}_{\min}$ and $\widebar{C}_{max}$  
are asymptotically independent and jointly converges in distribution to $(-Z_1, Z_2)$, where $Z_1$ and $Z_2$ are independent random variables with the reversed Gumbel distribution $D$ as defined in Theorem~\ref{thm1}. 
\end{theorem}

As mentioned in Remark 10 of Cai et al.~\cite{CFJ13},  the coherence of high-dimensional random matrix  $(X_1^T,\ldots, X_n^T)^T$ is defined as 
\[
L_{n,p}=\max_{1\le i< j\le  n}|X_i^TX_j|.
\]

In the study of the coherence of a data matrix formed by random samples from multivariate normal distributions, Cai and Jiang~\cite{CaiJiang2012} investigated the largest magnitude of
the squared Pearson correlation coefficients.   
When the data matrix consists of $p+1$ independent data points from an $n$-dimensional multivariate normal distribution with independent coordinates,  the coherence under investigation in Cai and Jiang~\cite{CaiJiang2012} has the same distribution as the following statistic 
\[
\widebar{L}_{n,p}=\max_{1\le i< j\le  n}(X_i^TX_j)^2.
\]
For more general case, see Cai and Jiang~\cite{CaiJiang2011}.

The limiting distributions for both $L_{n,p}$ and  $\widebar{L}_{n,p}$  can be obtained from Theorem~\ref{thm2}.

\begin{corollary}\label{cor2}
Assume $p=p_n\to\infty$ as $n\to\infty$. Then
\begin{equation}\label{limitingLnp}
\frac{1}{\alpha_nt_n}\big(L_{n,p}
-\sqrt{1-t_n^2}\big)+\ln (p_n-1)(1-t_n^2)\td \Lambda_1
\end{equation}
and
\begin{equation}\label{limitingL-bar}
\frac{p_n-1}{t_n^2}\left\{(\widebar{L}_{n,p}
-(1-t_n^2)\right\}+\ln (p_n-1)(1-t_n^2)\td \Lambda_1,
\end{equation}
where $\Lambda_1(x)=\Lambda^2(x)=\exp\left\{-\frac{1}{\sqrt{2\pi}}e^{-x/2}\right\}$, $x\in \RR$ is a Gumbel distribution.
    
\end{corollary}

Before concluding this section,  we compare our results in this paper briefly with those in Cai et al.~\cite{CFJ13} and Cai and Jiang~\cite{CaiJiang2012} and offer some comments. 
In Cai et al.~\cite{CFJ13} and Cai and Jiang~\cite{CaiJiang2012}, they investigated the limiting distributions for some functions of $\Theta_{\min}$ and $\Theta_{\max}$  and statistic $\widebar{L}_{n,p}$ under condition $\lim\limits_{n\to\infty}\frac{\ln n}{p}=\beta\in [0,\infty]$ and obtained their limiting distributions separately for three different cases, including $\beta=0$, $\beta\in (0,\infty)$, or $\beta=\infty$ and additional
conditions which reflect the relative convergence rate of $n$ and $p$. Although their limiting distributions are still Gumbel or reversed Gumbel, but their limits contain different centering or scaling constants since   they choose different normalization constants 
for different cases.  Consequently, they observed the so-called phase transition phenomena.  In our paper, those
transition phenomena disappear. For each statistic under investigation in this paper,  its normalization constants are appropriately selected so that its limit does not depend on how $n$ and $p$ grow.   Our results seem more convenient for applications.

\section{Proofs}\label{proofs}

We first introduce some notations used in this section. For two sequences of real numbers $\{a_n\}$ and $\{b_n\}$,  $a_n\sim b_n$,  $a_n=o(b_n)$, and  $a_n=O(b_n)$ 
stand for $\lim_{n\to\infty}(a_n/b_n)=1$, $\lim_{n\to\infty}(a_n/b_n)=0$ and $\limsup_{n\to\infty}|a_n/b_n|<\infty$, respectively. 
For any two sequences of random variables $\{U_n\}$ and $\{V_n\}$,
$U_n=o_p(V_n)$ implies $U_n/v_n$ converges in probability to zero, that is, $\lim_{n\to\infty}\PP(|U_n/V_n|>\varepsilon)=0$ for any $\varepsilon>0$, and
$U_n=O_p(V_n)$ implies $U_n/v_n$ is tight or equivalently
$\lim_{\varepsilon\to\infty}\limsup_{n\to\infty}\PP(|U_n/V_n|>\varepsilon)=0$.

Nest, we introduce a few lemmas before the proofs for the main results in the paper.

\begin{lemma}[Lemma 12 in Cai et al.~\cite{CFJ13} ]\label{recursion}
Let $p\geq 2$. Then,

\noindent (i) $\{\Theta_{ij};\, 1\leq i < j\leq n\}$ are pairwise independent and identically distributed with  density function
\begin{equation}\label{volume}
h(\theta) =\frac{1}{\sqrt{\pi}}\frac{\Gamma(\frac{p}{2})}{\Gamma(\frac{p-1}{2})}\cdot(\sin \theta)^{p-2},\ \ \ \theta \in [0, \pi].
\end{equation}

\noindent (ii) If ``$\Theta_{ij}$" in (i) is replaced by ``$\pi -\Theta_{ij}$", the conclusion in (i) still holds.
\end{lemma}

\begin{lemma}\label{tail probab} Assume $p=p_n\to \infty$ as $n\to\infty$, and  $\{\theta_n\}$ is a sequence of constants in $(0,\pi/2)$ such that $(p_n-1)(\cos\theta_n)^2\to\infty$. Then
\begin{equation}\label{low-tail}
\PP(\Theta_{12}\le \theta_n)=\int^{\theta_n}_0h(\theta)d\theta=\frac{1+o(1)}{\sqrt{2\pi}}\cdot\frac{(\sin \theta_n)^{p_n-1}}{\sqrt{p_n}\cos \theta_n}
\end{equation}
as $n\to\infty$. Furthermore, we have
\begin{equation}\label{upp-tail}
\PP(\pi-\Theta_{12}\le \theta_n)=\int^{\theta_n}_0h(\theta)d\theta=\frac{1+o(1)}{\sqrt{2\pi}}\cdot\frac{(\sin \theta_n)^{p_n-1}}{\sqrt{p_n}\cos \theta_n}
\end{equation}
    
\end{lemma}

\begin{proof}
Set
\[
I_n=\int^{\theta_n}_0(\sin \theta)^{p-2}d\theta.
\]
Then by integrating by parts
\begin{eqnarray*}
I_n&=&\frac{1}{p_n-1}\int^{\theta_n}_0\frac{[(\sin \theta)^{p_n-1}]'}{\cos \theta}d\theta\\
&=&\frac{1}{p_n-1}\frac{(\sin \theta)^{p_n-1}}{\cos \theta}\Big|^{\theta_n}_0-\frac{1}{p_n-1}\int^{\theta_n}_0\frac{(\sin \theta)^{p_n-1}\sin\theta}{(\cos \theta)^2}d\theta\\
&=&\frac{1}{p_n-1}\frac{(\sin \theta_n)^{p_n-1}}{\cos \theta_n}-\int^{\theta_n}_0(\sin \theta)^{p_n-2}\frac{(\sin\theta)^2}{(p_n-1)(\cos \theta)^2}d\theta.
\end{eqnarray*}
Since $\cos \theta$ is decreasing in $\theta\in (0, \pi/2)$,
\[
\sup_{\theta\in (0, \theta_n)}\frac{(\sin \theta)^2}{(p_n-1)(\cos \theta)^2}\le \frac{1}{(p_n-1)(\cos \theta_n)^2}\to 0 
\]
as $n\to\infty$.
We obtain that
\[
I_n=\frac{1}{p_n-1}\frac{(\sin \theta_n)^{p_n-1}}{\cos \theta_n}+o(I_n),
\]
which implies 
\[
I_n=\frac{1}{p_n-1}\frac{(\sin \theta_n)^{p_n-1}}{\cos \theta_n}(1+o(1)).
\]

By using the Stirling formula we have 
\begin{equation}\label{white}
\frac{\Gamma(\frac{p_n}{2})}{\Gamma(\frac{p_n-1}{2})}= \sqrt{\frac{p_n-1}{2}}(1+o(1))
\end{equation}
as $p_n\to\infty$; see also equation (17) in Cai et al.~\cite{CFJ13}. Then it follows from Lemma~\ref{recursion} (i) with $p=p_n$ that
\[
\PP(\Theta_{12}\le \theta_n)=\int^{\theta_n}_0h(\theta)d\theta=
\frac{1}{\sqrt{\pi}}\frac{\Gamma(\frac{p_n}{2})}{\Gamma(\frac{p_n-1}{2})}\cdot I_n=
\frac{1+o(1)}{\sqrt{2\pi}}\cdot\frac{(\sin \theta_n)^{p_n-1}}{\sqrt{p_n-1}\cos \theta_n},
\]
proving \eqref{low-tail}.  \eqref{upp-tail} follows from Lemma~\ref{recursion} (ii) and \eqref{low-tail}.  This completes the proof of the lemma.
\end{proof}

\begin{lemma}\label{lem3}
Define for each fixed $s\in \RR$
\[
\Delta_n(s)=\frac{\ln (p_n-1)(1-t_n^2)}{p_n-1}+\frac{s}{p_n-1} 
\]
and set
\begin{equation}\label{theta}
\theta_n(s)=\arcsin (t_n)+\frac{t_n}{2\sqrt{1-t_n^2}}\Delta_n(s). 
\end{equation}
Then $0<\theta_n(s)<\pi/2$ for all large $n$,
\begin{equation}\label{lower}
\lim_{n\to\infty}\frac{n(n-1)}{2}\PP\big(\Theta_{12}\le \theta_n(s)\big)=\frac{e^{s/2}}{2\sqrt{2\pi}}.
\end{equation}   
and
\begin{equation}\label{upper}
\lim_{n\to\infty}\frac{n(n-1)}{2}\PP\big(\pi-\Theta_{12}\le \theta_n(s)\big)= \frac{e^{s/2}}{2\sqrt{2\pi}}.
\end{equation}
\end{lemma}

\begin{proof}

Note that $\arcsin$ is the inverse of the sine function.  Since $\sin x$ is strictly increasing in  $(0, \pi/2)$,  $\arcsin t$ is well defined for $t\in (0,1)$ and $\arcsin (t)\in (0, \pi/2)$ for $t\in (0,1)$.   

Define $\delta_n=\min(\ln n, \sqrt{p_n-1})$.  Then 
\begin{equation}\label{delta}
    \delta_n\to\infty, ~~\frac{\delta_n}{p_n-1}\le \frac1{\sqrt{p_n-1}}\to 0 
\end{equation}
as $n\to\infty$.

First, we show that

\begin{equation}\label{bound}
\delta_n< (p_n-1)(1-t_n^2)<p_n-1
\end{equation}
for all large $n$.
The second inequality of \eqref{bound} is trivial since $t_n\in (0,1)$.

Observe that
\[
(p_n-1)(1-t_n^2)=(p_n-1)\left\{1-\exp\left(-
\frac{4\ln n}{p_n-1}\right)\right\} \ge (p_n-1)\left\{1-\exp\Big(-
\frac{4\delta_n}{p_n-1}\Big)\right\}
>\delta_n 
\]
for all large $n$ due to the fact that $\lim_{n\to\infty}\delta_n/(p_n-1)=0$ from \eqref{delta}. This completes the proof of the first inequality of \eqref{bound}. 

For any $s\in \RR$,  we have from \eqref{bound} and \eqref{delta} that
\begin{equation}\label{1st}
\Delta_n(s)=O\left(\frac{\ln (p_n-1)(1-t_n^2)}{p_n-1}\right)\to 0
\end{equation}
and
\begin{equation}\label{2nd}
\frac{\Big(\Delta_n(s)\Big)^2}{1-t_n^2}=\frac{\Big((p_n-1)\Delta_n(s)\Big)^2}{(p_n-1)^2(1-t_n^2)}=\frac{O\Big(\big[\ln (p_n-1)(1-t_n^2)\big]^2\Big)}{(p_n-1)^2(1-t_n^2)}=o(\frac1{p_n-1})
\end{equation}
as $n\to\infty$. The above two equations will be used in the proofs of the following expansions for $\sin(\theta_n(s))$ and $\cos(\theta_n(s))$:
\begin{equation}\label{sin-p}
\big(\sin(\theta_n(s)\big)^{p_n-1}=\frac{\sqrt{(p_n-1)(1-t_n^2)}}{n^2}e^{s/2}\big(1+o(1)\big)
\end{equation}
and
\begin{equation}\label{cos-theta}
\cos(\theta_n(s))=\sqrt{1-t_n^2}(1+o(1)).
\end{equation}

First, observe that
\[
\sin (\arcsin(t_n))=t_n, ~~\cos(\arcsin(t_n))=\sqrt{1-t_n^2}.
\]
Then it follows from Taylor's expansion that
\begin{eqnarray*}
\sin(\theta_n(s))&=&\sin (\arcsin(t_n))+\cos(\arcsin(t_n))\frac{t_n}{2\sqrt{1-t_n^2}}\Delta_n(s)+O\Big(\big[\frac{t_n}{\sqrt{1-t_n^2}}\Delta_n(s)\big]^2\Big)\\
&=&t_n+\frac12t_n\Delta_n(s)+O\Big(\big[\frac{t_n}{\sqrt{1-t_n^2}}\Delta_n(s)\big]^2\Big)\\
&=&t_n\Big(1+\frac12\Delta_n(s)+O(\frac{(\Delta_n(s))^2}{1-t_n^2})\Big).
\end{eqnarray*}
Therefore, we have
\begin{eqnarray*}
\big(\sin(\theta_n(s)\big)^{p_n-1}
&=&\exp\Big\{(p_n-1)\ln t_n+(p_n-1)\ln \Big(1+\frac12\Delta_n(s)+O(\frac{(\Delta_n(s))^2}{1-t_n^2})\Big)\Big\}\\
&=&\exp\Big\{-2\ln n+\frac12(p_n-1)\Delta_n(s)+O(\frac{(\Delta_n(s))^2}{1-t_n^2})\Big\}\\
&=&\exp\Big\{-2\ln n+\frac12\ln(p_n-1)(1-t_n^2)+\frac{s}2+o(1)\Big\},
\end{eqnarray*}
proving \eqref{sin-p}. Similarly, we have
\begin{eqnarray*}
\cos(\theta_n(s))&=&\cos (\arcsin(t_n))+O\big(\frac{t_n}{\sqrt{1-t_n^2}}\Delta_n(s)\big) \\
&=&\sqrt{1-t_n^2}\Big(1+O\big(\frac{1}{1-t_n^2}\Delta_n(s)\big)\Big) \\
&=&\sqrt{1-t_n^2}\Big(1+o\big(1\big)\Big)
\end{eqnarray*}
and obtain \eqref{cos-theta}.

Now we will show $0<\theta_n(s)<\pi/2$ for all large $n$.   From definition \eqref{theta}, we have
\[
t_n(1+ \frac{\Delta_n(s)}
{2\sqrt{1-t_n^2}})\le \theta_n(s)\le \frac{\pi}{2}+
\frac{t_n\Delta_n(s)}
{2\sqrt{1-t_n^2}}.
\]
In establishing the first inequality above, we have used a simple fact that $t_n\le \arcsin(t_n)$.   From \eqref{2nd} we conclude that $0<\theta_n(s)<\frac{2\pi}{3}$ for all large $n$. Moreover, it follows from \eqref{cos-theta} that  
$\theta_n(s)>0$ for all large $n$, which implies
$0<\theta_n(s)<\pi/2$.

From \eqref{cos-theta} and \eqref{bound},  $p_n(\cos \theta_n)^2\to\infty$, and thus we can apply Lemma~\ref{tail probab} and have
\[
\frac{n(n-1)}{2}\PP(\Theta_{12}\le \theta_n(s))=
\frac{n^2(1+o(1))}{2\sqrt{2\pi}}\cdot\frac{(\sin \theta_n(s))^{p_n-1}}{\sqrt{p_n-1}\cos \theta_n(s)},
\]
which converges to $\frac{e^{s/2}}{2\sqrt{2\pi}}$ from \eqref{sin-p} and \eqref{cos-theta}. That is,
\eqref{lower} holds.  \eqref{upper} holds as well since $\pi-\Theta_{12}$ and $\Theta_{12}$ have the same distribution from Lemma~\ref{recursion} (ii).
\end{proof}

Conclusion:

\[
\lim_{n\to\infty}\PP\Big(\bar\Theta_{\min}\le \theta_{n}(y)\Big)=\exp(-\frac{e^{y/2}}{2\sqrt{2\pi}})
\]

We also need the following result on Poisson approximation, see e,g. Theorem 1 in \cite{AGG}.

 \begin{lemma}\label{lem:Poisson}
        Let $T$ be an index set, and for $\alpha\in T $,
        let $Z_{\alpha}$ be a Bernoulli random variable with
        $p_{\alpha}=\PP(Z_{\alpha}=1)=1-\PP(Z_{\alpha}=0)$.
        For each $ \alpha\in T$, let $T_{\alpha}$ be a subset of $T$ with $ \alpha\in T_{\alpha}$. Set 
        \[
        S=\sum_{\alpha\in T}Z_{\alpha},\,\, \lambda=\mathbb E S=\sum_{\alpha\in T}p_{\alpha}\in (0, \infty).
        \]
        Assume $W$ is a Poisson random variable with mean $\mathbb E W=\lambda$,  
      and denote $d_{\mathrm{TV}}(\mathcal L(S),\mathcal L(W))$ as the total variation distance between the distributions of $S$ and $W$. Then
        \[
        d_{\mathrm{TV}}(\mathcal L(S),\mathcal L (Z))\leq 2(b_1+b_2+b_3),
        \]
   where 
    \[
    b_1=\sum_{\alpha\in T}\sum_{\beta\in T_{\alpha}}p_{\alpha}p_{\beta}, ~b_2=\sum_{\alpha\in T}\sum_{\alpha\neq\beta\in T_{\alpha}}\mathbb{E}[Z_{\alpha}Z_{\beta}],~b_3=\sum_{\alpha\in T}\mathbb{E}|\mathbb{E}[Z_{\alpha}|\sigma(Z_{\beta},\beta\not\in T_{\alpha})]-p_{\alpha}|,
    \]
 and $\sigma(Z_{\beta},\beta\not\in T_{\alpha})$ is the $\sigma$-algebra generated by $\{Z_{\beta},\beta\not\in T_{\alpha}\}$. 
    \end{lemma}

\vspace{10pt}

\begin{proof}[Proof of Theorem~\ref{thm1}]

Define an index set $T_n=\big\{ (i,j):  1\le i<j\le n\big\}$. 
For each $(i,j)\in T_n$,  define $T_{ij}=\{(k,l): 
(k, l)\in T_n, \{i,j\}\cap\{k,l\}\ne \emptyset\}$.

For any fixed $x,y\in \RR$,  define
set $Z_{ij}=\mathbf{I}\big(\{\Theta_{ij}\le \theta_n(-x)\}\cup\{\pi-\Theta_{ij}\le \theta_n(-y)\}\big)$ for $1\le i<j\le n$, where $\mathbf{I}(A)$ denotes the indicator function for any set $A$.  Define  $S_n=\sum_{1\le i<j\le n}Z_{ij}$.

We need to use the the following facts in our proof.

\noindent \textbf{Fact 1}. $\{\Theta_{\alpha}, ~1\le i<j\le n\}$ are pairwise independent;

\noindent \textbf{Fact 2}. $\Theta_{ij}$ and $\{\Theta_{kl}: ~(k,l)\notin T_{ij}\}$ are independent for each $(i,j)\in T_n$.

\noindent \textbf{Fact 3}. $\theta_n(-x)<\pi-\theta_n(-y)$ for all large $n$.

Facts 1 and 2 are from Cai et al.~\cite{CFJ13}  .
Fact 3 follows from Lemma~\ref{lem3} that $0<\theta_n(-x)<\pi/2$ and $0<\theta_n(-y)<\pi/2$ for all large $n$.

We proceed to show that
$S_n$ converges in distribution to a Poisson random variable with mean $\frac{e^{x/2}}{2\sqrt{2\pi}}+\frac{e^{y/2}}{2\sqrt{2\pi}}$
by using Lemma~\ref{lem:Poisson}.

First, we note that  $\{\Theta_{ij}\le \theta_n(-x)\}$ and $\{\pi-\Theta_{ij}\le \theta_n(y)\}$
are disjoint for all large $n$ due to Fact 3. Therefore, we have from Lemma~\ref{lem3} that
\begin{eqnarray*}
&&\PP(\{\Theta_{12}\le \theta_n(-x)\}\cup\{\pi-\Theta_{12}\le \theta_n(-y)\})\\
&=&
\PP(\Theta_{12}\le \theta_n(-x))+\PP(\pi-\Theta_{12}\le \theta_n(-y))\\
&=&\frac{2}{n(n-1)}\big(\frac{e^{-x/2}}{2\sqrt{2\pi}}+\frac{e^{-y/2}}{2\sqrt{2\pi}}\big)\big(1+o(1)\big)
\end{eqnarray*}
and thus
\begin{eqnarray*}
\lambda_n:=\EE(S_n)&=&\frac{n(n-1)}{2}\PP(\{\Theta_{12}\le \theta_n(-x)\}\cup\{\pi-\Theta_{12}\le\theta_n(-y)\})\\
&=&\big(\frac{e^{-x/2}}{2\sqrt{2\pi}}+\frac{e^{-y/2}}{2\sqrt{2\pi}}\big)\big(1+o(1)\big).
\end{eqnarray*}

We will calculate $b_1$, $b_2$ and $b_3$ in Lemma~\ref{lem:Poisson}. We see that 
\[
b_1=\binom{n}{2}(n-1)\Big(\PP(\{\Theta_{12}\le \theta_n(-x)\}\cup\{\pi-\Theta_{12}\le\theta_n(-y)\})\Big)^2=O(\frac1n).
\]
Write $A_{ij}=\{\Theta_{ij}\le \theta_n(-x)\}\cup\{\pi-\Theta_{ij}\le\theta_n(-y)\}$. Then $A_{12}$ and $A_{13}$ are independent from Fact 1. We obtain 
\[
b_2=\binom{n}{2}(n-2)\PP(A_{12}\cap A_{13})
=\binom{n}{2}(n-2)\PP(A_{12})\PP(A_{13})
\le b_1
\]
for all large $n$. We also have from Fact 2 that $b_3=0$.

Therefore,  we conclude from Lemma~\ref{lem:Poisson} that
$d_{\mathrm{TV}}(\mathcal L(S_n),\mathcal L (Z_n))\to 0$ as $n\to\infty$, where $Z_n$ is a Poisson random variable with mean $\lambda_n$.  Since $\lambda_n\to \frac{e^{-x/2}}{2\sqrt{2\pi}}+\frac{e^{-y/2}}{2\sqrt{2\pi}}$, it follows that $S_n$ converges in distribution to a Poisson distribution with mean
$\frac{e^{-x/2}}{2\sqrt{2\pi}}+\frac{e^{-y/2}}{2\sqrt{2\pi}}$.  In particular, it holds that
\[
\lim_{n\to\infty}\PP(S_n=0)=
\exp\Big\{-(\frac{e^{-x/2}}{2\sqrt{2\pi}}+\frac{e^{-y/2}}{2\sqrt{2\pi}})\Big\}=
\exp\Big\{-\frac{e^{-x/2}}{2\sqrt{2\pi}}\Big\}\exp\Big\{-\frac{e^{-y/2}}{2\sqrt{2\pi}}\Big\}.
\]

Finally, noting that
$\{-\widebar\Theta_{\min}\le x, -\widebar\Theta_{\max}\le y\}=\{S_n=0\}$,
we get
\[
\lim_{n\to\infty}\PP(-\widebar\Theta_{\min}\le x, -\widebar\Theta_{\max}\le y)=\exp\Big\{-\frac{e^{-x/2}}{2\sqrt{2\pi}}\Big\}\exp\Big\{-\frac{e^{-y/2}}{2\sqrt{2\pi}}\Big\}
\]
for all $x,y\in \RR$. The right-hand side is the product of two Gumbel distribution functions.  Hence,   $-\widebar\Theta_{\min}$ and $-\widebar\Theta_{\max}$
are asymptotically independent, which implies 
$\widebar\Theta_{\min}$ and $\widebar\Theta_{\max}$
are asymptotically independent as well, and their limits are a reversed Gumbel distribution
$D(x)=1-\exp(-\frac1{2\sqrt{2\pi}}e^{x/2})$, $x\in \RR$.
This completes the proof of Theorem~\ref{thm1}.
\end{proof}

\begin{proof}[Proof of Corollary~\ref{cor1}]

Part (a) of the corollary can be easily obtained by subtracting \eqref{thetamax} from \eqref{thetamin} and using the asymptotic independence of  
$\widebar{\Theta}_{\max}$ and $\widebar{\Theta}_{\min}$.  

To show part (b) of Corollary~\ref{cor1}, we have from \eqref{thetamax} from \eqref{thetamin} that 
\begin{equation}\label{expression1}
\Theta_{\min}-\arcsin(t_n)= \beta_n+\alpha_n\widebar{\Theta}_{\min}
\end{equation}
and
\begin{equation}\label{expression2}
\pi-\Theta_{\max}-\arcsin(t_n)=\beta_n+\alpha_n\widebar{\Theta}_{\max},
\end{equation}
where $\alpha_n$ and $\beta_n$ are defined in \eqref{alpha-beta}.
Now from \eqref{delta} and \eqref{bound} in the proof of Lemma~\ref{lem3} we conclude
\begin{equation}\label{ab-small}
   \alpha_n=o(\beta_n), ~~~\beta_n=o(t_n\sqrt{1-t_n^2}), ~~~\beta_n^2=o(t_n\alpha_n)
\end{equation}
as $n\to\infty$, which together with Theorem~\ref{thm1} imply that the right-hand sides of \eqref{expression1} and \eqref{expression2} converges in probability to zero.   This completes the proof of Corollary~\ref{cor1}. 
    \end{proof}

\begin{proof}[Proof of Theorem~\ref{thm2}]
    First, rewrite \eqref{expression1} and \eqref{expression2} as
\begin{equation}\label{theta1}
\Theta_{\min}= \arcsin(t_n)+ \beta_n+\alpha_n\widebar{\Theta}_{\min}
\end{equation}
and
\begin{equation}\label{theta2}
\Theta_{\max}= \pi-\arcsin(t_n)- \beta_n-\alpha_n\widebar{\Theta}_{\max}.
\end{equation}
Taking advantage of \eqref{theta1} and using Taylor's expansion,  we have
\begin{eqnarray*}
\cos(\Theta_{\min})&=&\cos(\arcsin(t_n)+\beta_n+\alpha_n\widebar{\Theta}_{\min})\\
&=&\cos(\arcsin(t_n))\cos((\beta_n+\alpha_n\widebar{\Theta}_{\min})
-\sin(\arcsin(t_n))\sin((\beta_n+\alpha_n\widebar{\Theta}_{\min}))\\
&=&\sqrt{1-t_n^2}\cos(\beta_n+\alpha_n\widebar{\Theta}_{\min})
-t_n\sin(\beta_n+\alpha_n\widebar{\Theta}_{\min})\\
&=&\sqrt{1-t_n^2}
-t_n(\beta_n+\alpha_n\widebar{\Theta}_{\min})+O_p\big((\beta_n+\alpha_n\widebar{\Theta}_{\min}
)^2\big)\\
&=&\sqrt{1-t_n^2}-t_n\beta_n -t_n\alpha_n\widebar{\Theta}_{\min}+O_p(\alpha_n^2+\beta_n^2)\\
&=&\sqrt{1-t_n^2}-t_n\beta_n -t_n\alpha_n\widebar{\Theta}_{\min}+o_p(t_n\alpha_n)
\end{eqnarray*}
from \eqref{ab-small}, 
which implies
\begin{eqnarray*}
\widebar{C}_{\min}&=&
\frac{1}{\alpha_nt_n}\big(\cos(\Theta_{\min})
-\sqrt{1-t_n^2}\big)+\ln (p_n-1)(1-t_n^2)\\
&=&\frac1{t_n\alpha_n}\Big(\cos(\Theta_{\min})-(\sqrt{1-t_n^2}-t_n\beta_n)\Big)\\
&=&-\widebar{\Theta}_{\min}+o_p(1)\\
&\td& \Lambda.
\end{eqnarray*}
Following the same lines as above, we obtain from \eqref{theta2} that
\begin{eqnarray*}
    \cos(\Theta_{\max})&=&-\cos(\arcsin(t_n)+\beta_n+\alpha_n\widebar{\Theta}_{\max})\\
&=&-\sqrt{1-t_n^2}+t_n\beta_n +t_n\alpha_n\widebar{\Theta}_{\max}+o_p(t_n\alpha_n),
\end{eqnarray*}
yielding
\[
\cos(\Theta_{\max})=\frac{1}{\alpha_nt_n}\big(\cos(\Theta_{\max})
+\sqrt{1-t_n^2}\big)-\ln (p_n-1)(1-t_n^2)
=\widebar{\Theta}_{\max}+o_p(1)\\
\td D.
\]
This completes the proof. 
\end{proof}

\begin{proof}[Proof of Corollary~\ref{cor2}]
Since  $\cos(t)$ is strictly decreasing in $t\in [0, 2\pi]$, and
\[
-1\le \cos(\Theta_{\max})\le X_i^TX_j\le \cos(\Theta_{\min})\le 1
\]
for $1\le i<j\le n$, we conclude  that
\[
L_{n,p}=\max\{|\cos(\Theta_{\min})|, |\cos(\Theta_{\max})|\}.
\]
One can also easily verify that 
\[
\max\{|x|, |y|\}=\max\{x, -y\}~~~ \text{ if } x\ge y,
\]
from which we obtain
\[
L_{n,p}=\max\big\{\cos(\Theta_{\min}), -\cos(\Theta_{\max})\big\}.
\]
and have
\begin{eqnarray*}
&&\frac{1}{\alpha_nt_n}\big(L_{n,p}
-\sqrt{1-t_n^2}\big)+\ln (p_n-1)(1-t_n^2)\\
&=&
\frac{1}{\alpha_nt_n}\Big(\max\big\{\cos(\Theta_{\min}), -\cos(\Theta_{\max})\big\}
-\sqrt{1-t_n^2}\Big)+\ln (p_n-1)(1-t_n^2)\\
&=&\max\{\widebar{C}_{\min}, -\widebar{C}_{\max}\}.
\end{eqnarray*}

Now we apply Theorem~\ref{thm2} and conclude that
$(\widebar{C}_{\min}, \widebar{C}_{\max})$ converges weakly to $(-Z_1, -Z_2)$.
From the continuous mapping theorem, we have
$\max\{\widebar{C}_{\min}, -\widebar{C}_{\max}\}$ converges in distribution to $\max\{-Z_1,-Z_2\}$.  Since the distribution  function for the maximum of two independent random variables with Gumbel distribution $\Lambda$ is equal to $\Lambda_1=\Lambda^2$,  we have
$\max\{\widebar{C}_{\min}, -\widebar{C}_{\max}\}\td \Lambda^2$. 
This completes the proof of  \eqref{limitingLnp}. 

To show \eqref{limitingL-bar}, we will use \eqref{limitingLnp}.
Since $\widebar{L}_{n,p}=L_{n,p}^2$, we have
\begin{eqnarray*}
\widebar{L}_{n,p}-(1-t_n^2)&=&2\sqrt{1-t_n^2}(L_{n,p}-\sqrt{1-t_n^2})+(L_{n,p}-\sqrt{1-t_n^2})^2,
\end{eqnarray*}
and thus
\begin{equation}\label{bridge}
\frac{1}{2\alpha_nt_n\sqrt{1-t_n}}\Big(\widebar{L}_{n,p}-(1-t_n^2)\Big)=\frac1{\alpha_nt_n}(L_{n,p}-\sqrt{1-t_n^2})+
\frac{(L_{n,p}-\sqrt{1-t_n^2})^2}{2\alpha_nt_n\sqrt{1-t_n}}.
\end{equation}
Noting that $2\alpha_nt_n\sqrt{1-t_n}=\frac{t_n^2}{p_n-1}$ and $L_{n,p}-\sqrt{1-t_n^2}=O_p\big(\alpha_nt_n\ln (p_n-1)(1-t_n^2)\big)$ from \eqref{limitingLnp},
we get
\[
\frac{(L_{n,p}-\sqrt{1-t_n^2})^2}{2\alpha_nt_n\sqrt{1-t_n}}=O_p\Big(\frac{\big(\ln (p_n-1)(1-t_n^2)\big)^2}{(p_n-1)(1-t_n^2)}\Big)=o_p(1)
\]
from \eqref{delta} and \eqref{bound}.  By combining \eqref{limitingLnp}, \eqref{bridge} and the identity 
$2\alpha_nt_n\sqrt{1-t_n}=\frac{t_n^2}{p_n-1}$,  we have 
\[
\frac{p_n-1}{t_n^2}\Big(\widebar{L}_{n,p}-(1-t_n^2)\Big)+\ln(p_n-1)(1-t_n^2)=\frac1{\alpha_nt_n}(L_{n,p}-\sqrt{1-t_n^2})+\ln (p_n-1)(1-t_n^2)+
o_p(1)
\]
which converges in distribution to $\Lambda_1$.  The proof of Corollary~\ref{cor2} is completed.
\end{proof}



\end{document}